\documentclass{article}
\usepackage[utf8]{inputenc}
\usepackage{amsmath,amsfonts,amsthm,amssymb}
\usepackage{graphicx}
\usepackage[all]{xy}

\newtheorem{lm}{Lemma}
\newtheorem{theorem}{Theorem}
\newtheorem{proposition}{Proposition}
\newtheorem{corollary}{Corollary}
\theoremstyle{definition}
\newtheorem{definition}{Definition}
\newtheorem{remark}{Remark}
\newtheorem{example}{Example}

\title{A generalisation of the Burau representation and groups $G_{n}^{3}$ for classical braids}

\author{Vassily Olegovich Manturov, Igor Mikhailovich Nikonov}
\date{}

\newcommand{\Z}{{\mathbb Z}}
\newcommand{\R}{{\mathbb R}}

\begin{document}

\maketitle

\abstract{We consider a certain modification of the group $G^3_n$ which describes dynamics of point configurations, in particular braids, and define a representation of the modified $G^3_n$. The braid representation induced is powerful enough to detect the kernel of the Burau representation.}



\section{Introduction}

In \cite{Ma} the first named author introduced a family of groups $G_{n}^{k}$ for two positive integers $n,k$ and formulated the following principle:

{\em If dynamical systems describing a motion of $n$ particles has a nice codimension $1$ property governed exactly by $k$
particles then these dynamical systems have a topological invariant valued in $G_{n}^{k}$.}

The first main example is an invariant of braids valued in $G_{n}^{3}$. 
Most of examples constructed over the last 10 years describe invariants of fundamental groups of configuration spaces
of $n$ points, which can be considered as ``braid-like objects'' that means an object with the number of particles fixed.

In this paper, we deal with a certain modification of the group $G^3_n$ which still describes dynamics of point configurations, in particular, braids. We define a representation of the modified group $\hat G^3_n$. The braid representation induced by it is powerful enough to detect the kernel of the Burau representation. 

\subsection{Acknowledgements}
The authors are grateful to Liu Yangzhou for valuable remarks. 

The first author was supported by the grant 25-21-00884 Applied combinatorial geometry and topology.

\section{A generalisation of the Burau representation}

Artin's presentation of the braid group $B_n$ has $n-1$ generators $\sigma_1,\dots,\sigma_{n-1}$ and the relations
\begin{itemize}
    \item (far commutativity) $\sigma_i\sigma_j=\sigma_j\sigma_i$ for any $|i-j|>1$;
    \item (Artin relations) $\sigma_i\sigma_{i+1}\sigma_i=\sigma_{i+1}\sigma_i\sigma_{i+1}$ for $i=1,\dots,n-2$.
\end{itemize}

Denote the natural homomorphism from $B_n$ to the permutation group $\Sigma_{n}$ on the set $\{1,\dots,n\}$ by $\rho$:
\begin{equation}\label{eq:map_from_braids_to_permutations}
\rho(\sigma_i)=(i\ i+1),\quad 1\le i\le n-1.
\end{equation}

The kernel $PB_n=\ker\rho$ is called the \emph{pure braid group} on $n$ strands.

Consider a modification of the groups $G^k_n$ introduced in~\cite[Section 2.3]{FKM}.

\begin{definition}
    The group $\hat G^3_n$ is the group with generators $a_{ijk}$ where $i,j,k$ are different numbers in $\{1,\dots,n\}$, and relations
\begin{enumerate}
    \item $a_{kji}=a_{ijk}^{-1}$;
    \item $a_{ijk}a_{pqr}=a_{pqr}a_{ijk}$ if there are at least 5 different numbers among $i,j,k,p,q,r$;
    \item $a_{ijk}a_{ijl}a_{ikl}a_{jkl}=a_{jkl}a_{ikl}a_{ijl}a_{ijk}$.
\end{enumerate}
\end{definition}

The permutation group $\Sigma_n$ acts on $\hat G^3_n$ by renumbering the indices:
\begin{equation}\label{eq:permutation_action_Gkn}
\tau \left(\prod_p a_{ijk}\right)=\prod_p a_{\tau(i)\tau(j)\tau(k)},\quad \tau\in \Sigma_n.
\end{equation}
This action defines a group structure on $\Sigma_n\ltimes \hat G^3_n$ by the formula
\begin{equation}\label{eq:product_Gkn_permutation}
(\sigma_1, \omega_1)\cdot(\sigma_2,\omega_2)=(\sigma_1\sigma_2, \sigma_2(\omega_1)\omega_2),\quad (\sigma_i,\omega_i)\in \Sigma_n\ltimes \hat G^3_n.
\end{equation}

\begin{figure}[h]
\centering\includegraphics[width=0.3\textwidth]{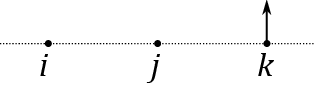}
\caption{The event $a_{ijk}$}
\label{fig:aijk}
\end{figure}

In~\cite{MN}, a homomorphism from $PB_n$ to $G^3_n$ was constructed. To do that, one looks at the braids as loops in the configuration space of $n$-point subsets of the plane. More concretely, a braid $\beta$ is given by a continuous map $x(t)=(x_1(t),...,x_n(t))$, $t\in[0,1]$, where $\{x_1(t),...,x_n(t)\}$ is an $n$-point subset of $\R^2$ for any $t$, and $\{x_1(0),...,x_n(0)\}=\{x_1(1),...,x_n(1)\}\subset S^1$. The corresponding element $\phi_n(\beta)$ of $G^3_n$ is constructed as follows: let $0<t_1<\cdots<t_p<1$ be the moments when three points in the subset $x(t)$ become collinear and let $x_{i_r}(t_r), x_{j_r}(t_r), x_{k_r}(t_r)$ be those points for the moment $t_r$, $r=1,\dots,p$. Then
\[
\phi_n(\beta)=\prod_{r=1}^p a_{i_rj_rk_r}.
\]

The homomorphism $\phi_n\colon PB_n\to G^3_n$ can be generalised to a homomorphism from $B_n$ to $\Sigma_n\ltimes \hat G^3_n$.

\begin{proposition}
The map $\phi_n$ whose values in the generators $\sigma_i$, $i=1,\dots,n-1$, are
\begin{equation}\label{eq:map_from_braid_to_Gkn}
\phi_n(\sigma_i)=\left(\rho(\sigma_i), a_{i-1,i+1,i}a_{i-2,i+1,i}\cdots a_{1,i+1,i}\cdot a_{n,i+1,i}\cdots a_{i+2,i+1,i}\right).
\end{equation}
induces a well defined homomorphism $\phi_n\colon B_n\to\Sigma_n\ltimes \hat G^3_n$.
\end{proposition}

\begin{figure}[h]
\centering\includegraphics[width=0.3\textwidth]{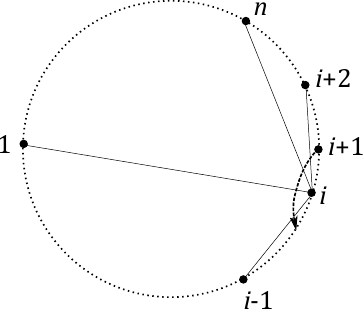}
\caption{Dynamics of the generator $\sigma_i$}
\label{fig:sigmadynamics}
\end{figure}

\begin{proof}
By considering the dynamics for the Artin generators (Fig.~\ref{fig:sigmadynamics}), we get the formula~\eqref{eq:map_from_braid_to_Gkn}. The correctedness of the map $\phi_n$ is proved analogous to~\cite[Theorem 2]{Ma}.
\end{proof}

It can be said that the proof of invariance follows from the fact that THE GENERATORS CORRESPOND TO CODIMENSION 1, and THE RELATIONS CORRESPOND TO CODIMENSION 2.

\begin{corollary}
    The map  $\phi_n$ induces a homomorphism from $PB_n$ to $\hat G^3_n$.
\end{corollary}
\begin{proof}
 For any $\beta\in PB_n$, we have $\rho(\beta)=1$. Then $\phi_n$ maps $PB_n$ in the subgroup $1\times \hat G^3_n\subset \Sigma_n\ltimes \hat G^3_n$ which is isomorphic to $\hat G^3_n$.
\end{proof}

Let $A=\Z[t_1^\pm,\dots,t_n^\pm,s_1^\pm,\dots,s_n^\pm]$. Consider a free module 
\[
V=A\langle x_{ij}\mid i,j\in\{1,\dots,n\}, i\ne j\rangle
\]
of rank $n(n-1)$. Define an endomorphism $\rho_{ijk}\in End_A(V)$ using the formula
\[
\rho_{ijk}(x_{pq})=\left\{\begin{array}{cl}
    t_ix_{ij}+(1-t_i)x_{ik}, & (p,q)=(i,j), \\
    t_k^{-1}x_{kj}+(1-t_k^{-1})x_{ki}, & (p,q)=(k,j), \\
    s_j x_{jk}, & (p,q)=(j,k),\\
    s_j^{-1} x_{ji}, & (p,q)=(j,i),\\
    x_{pq}, & \mbox{otherwise}.
\end{array}\right.
\]

\begin{theorem}
    The map $a_{ijk}\mapsto \rho_{ijk}$ defines a representation of the group $\hat G^3_n$ on $V$.
\end{theorem}
\begin{proof}
    The relations 1 and 2 follow from the definition. The correctness of the relation 3 is verified directly (it is enough to consider the case $(i,j,k,l)=(1,2,3,4)$).
\end{proof}

\begin{corollary}
    The composition $\rho\circ\phi_n\colon PB_n\to End_A(V)$ is a representation of the pure braid group on the free module $V$.
\end{corollary}

\begin{example}
Let $n=5$ and $\beta\in PB_5$ be the element in the kernel of the Burau representation considered in~\cite{Big}. The braid $\beta$ is defined by the formula 
\[
\beta = [\psi_1^{-1}\sigma_4\psi_1,\psi_2^{-1}\sigma_4\sigma_3\sigma_2\sigma_1^2\sigma_2\sigma_3\sigma_4\psi_2],
\]
where $\psi_1=\sigma_3^{-1}\sigma_2\sigma_1^2\sigma_2\sigma_4^3\sigma_3\sigma_2$ and $\psi_2=\sigma_4^{-1}\sigma_3\sigma_2\sigma_1^{-2}\sigma_2\sigma_1^2\sigma_2^2\sigma_1\sigma_4^5$. Computer calculations show that the matrix $\rho\circ\psi_5(\beta)\in M_{20}(A)$ is not the identity. In particular, when $t_1=-1$ and $t_2=\cdots=t_5=s_1=\cdots=s_5=1$, the corner matrix element is
\[
\langle x_{12}|\rho\circ\psi_5(\beta)|x_{12}\rangle=-399.
\]
Thus, the constructed representation is not weaker than the Burau representation.
\end{example}

\begin{example}
    Consider the braid $\beta$ of the previous example as an element in the group $PB_6$. Liu Yangzhou~\cite{Yangzhou} proved that $\beta$ lies in the kernel of representation $\tilde\rho\circ\psi\circ f_d\circ p_6$, $d\ge 2$, considered in~\cite{MN23}. On the other hand, the matrix $\rho\circ\psi_6(\beta)\in M_{30}(A)$ is not the identity when $t_1=s_1=-1$ and $t_2=\cdots=t_6=s_2=\cdots=s_6=1$ (the corner matrix element is $\langle x_{12}|\rho\circ\psi_6(\beta)|x_{12}\rangle=-399$). Hence, the representation $\rho\circ\psi_6$ is stronger than the representation $\tilde\rho\circ\psi\circ f_d\circ p_6$.
\end{example}

\begin{remark}
    By specializing the values of the parameters $t_i,s_i$ in $\mathbb R$, we get a representation of the group $\hat G^3_n$ in the space $V_{\mathbb R}=\mathbb R\langle x_{ij}\mid i,j\in\{1,\dots,n\}, i\ne j\rangle$.  
    
    On the other hand, the space $V_{\mathbb R}$ splits into a direct sum of the subspaces $V_{sym}=\mathbb R\langle \frac{x_{ij}+x_{ji}}2\mid i\ne j\rangle$ and $V_{alt}=\mathbb R\langle \frac{x_{ij}-x_{ji}}2\mid i\ne j\rangle$. In~\cite[Section 2.3]{FKM}, several representations of $\hat G^3_n$ on $V_{sym}$ were defined. One can extend those representations to $V_{\mathbb R}$ by fixing a representation of $\hat G^3_n$ on $V_{alt}$ (for example, the trivial one). As can be seen, the representations obtained are different from those induced by the specializations of $\rho$. This indicates the existence of a more general representation that would encompass all of the above mentioned cases, but this representation has yet to be found. 
\end{remark}

\end{document}